\documentclass{amsart}

\usepackage{amsmath,amssymb,amsfonts}
\newtheorem{theorem}{Theorem}[section]
\newtheorem{lemma}[theorem]{Lemma}

\theoremstyle{definition}
\newtheorem{definition}[theorem]{Definition}
\newtheorem{example}[theorem]{Example}
\newtheorem{coro}[theorem]{Corollary}
\newtheorem{conj}[theorem]{Conjecture}

\theoremstyle{remark}
\newtheorem{remark}[theorem]{Remark}

\numberwithin{equation}{section}

\def \v{\mathbf{v}}
\def \u{\mathbf{u}}
\def \i{\mathbf{i}}
\def \A{\mathcal{A}}
\def \L{\mathcal{L}}
\def \Q{\mathcal{Q}}

\def \la{\lambda}
\def \I{\mathcal{I}}
\def \F{\mathbf{F}}
\def \W{\mathbf{W}}
\def \Dg{\mathfrak{D}}

\def \S{\mbox{Spec}}
\def \Tr{\mbox{Tr}}
\begin{document}

\title[Spectral Symmetry of Nonnegative Tensors and Hypergraphs]{The Spectral Symmetry of Weakly Irreducible Nonnegative Tensors and Connected Hypergraphs}

\author[Y.-Z. Fan]{Yi-Zheng Fan}
\address{School of Mathematical Sciences, Anhui University, Hefei 230601, P. R. China}
\email{fanyz@ahu.edu.cn}
\thanks{The first author is the corresponding author.
The first and the second authors were supported by National Natural Science Foundation of China \#11371028.
The third author was supported by National Natural Science Foundation of China \#11401001.}

\author[T. Huang]{Tao Huang}
\address{School of Mathematical Sciences, Anhui University, Hefei 230601, P. R. China}
\email{huangtao@ahu.edu.cn}

\author[Y.-H. Bao]{Yan-Hong Bao}
\address{School of Mathematical Sciences, Anhui University, Hefei 230601, P. R. China}
\email{baoyh@ahu.edu.cn}

\author[C.-L. Zhuan-Sun]{Chen-Lu Zhuan-Sun}
\address{School of Mathematical Sciences, Anhui University, Hefei 230601, P. R. China}
\email{zhuansuncl@163.com}

\author[Y.-P. Li]{Ya-Ping Li}
\address{School of Mathematical Sciences, Anhui University, Hefei 230601, P. R. China}
\email{18856961415@163.com}

\subjclass[2000]{Primary 15A18, 05C65; Secondary 13P15, 05C15}



\keywords{Tensor, hypergraphs, adjacency tensor, spectral symmetry, coloring}

\begin{abstract}
Let $\A$ be a weakly irreducible nonnegative tensor with spectral radius $\rho(\A)$.
Let $\Dg$ (respectively, $\Dg^{(0)}$) be the set of normalized diagonal matrices
  arising from the eigenvectors of $\A$ corresponding to the eigenvalues with modulus $\rho(\A)$ (respectively, the eigenvalue $\rho(\A)$).
It is shown that $\Dg$ is an abelian group containing $\Dg^{(0)}$ as a subgroup, which acts transitively on the set $\{e^{\i \frac{2 \pi j}{\ell}}\A:j =0,1, \ldots,\ell-1\}$,
   where $|\Dg/\Dg^{(0)}|=\ell$ and $\Dg^{(0)}$ is the stabilizer of $\A$.
The spectral symmetry of $\A$ is characterized by the group $\Dg/\Dg^{(0)}$, and $\A$ is called spectral $\ell$-symmetric.
 We obtain the structural information of $\A$ by analyzing the property of $\Dg$, especially for connected hypergraphs we get
some results on the edge distribution and coloring.
If moreover $\A$ is symmetric, we prove that $\A$ is spectral $\ell$-symmetric if and only if it is $(m,\ell)$-colorable.
We characterize the spectral $\ell$-symmetry of a tensor by using its generalized traces,
and show that for an arbitrarily given integer $m \ge 3$ and each positive integer $\ell$ with $\ell \mid m$, there always exists an $m$-uniform hypergraph $G$ such that $G$ is spectral $\ell$-symmetric.
\end{abstract}

\maketitle


\section{Introduction}
A real {\it tensor} (also called {\it hypermatrix}) $\A=(a_{i_{1} i_2 \ldots i_{m}})$ of order $m$ and dimension $n$ refers to a
  multidimensional array with entries $a_{i_{1}i_2\ldots i_{m}}\in \mathbb{R}$
  for all $i_{j}\in [n]:=\{1,2,\ldots,n\}$ and $j\in [m]$.
Surely, if $m=2$, then $\A$ is a square matrix of dimension $n$.
A {\it subtensor} of $\A$ is a multidimensional array with entries $a_{i_{1}i_2\ldots i_{m}}$ such that $i_j \in S_j \subseteq [n]$ for some $S_j$'s and $j \in [m]$,
denoted by $\A[S_1|S_2|\cdots|S_m]$.
Let $\rho(\A)$ be the spectral radius of $\A$, and $\S(\A)$ be the spectrum of $\A$.
The circle centered at the origin of the complex plane with radius $\rho(\A)$ is called the {\it spectral circle} of $\A$.

By the famous Perron-Frobenius theorem, for a nonnegative irreducible matrix $\A$ of dimension $n$, if it has $k$ eigenvalues with modulus
 $\rho(\A)$, then those $k$ eigenvalues are equally distributed on the spectral circle, i.e.
 they are $\rho(\A)e^{\i \frac{2 \pi j}{k}}$, $j =0,1, \ldots, k-1$.
Furthermore, the spectrum of $\A$ keeps invariant under a rotation of angle $\frac{2 \pi }{k}$ of the complex plane, i.e. $\S(A)=e^{\i \frac{2 \pi}{k}}\S(A)$.
Under the above spectral symmetry, $\A$ has a cyclic structure via a permutation matrix $P$, i.e.
\[\label{mat-stru}
P^T \A P =\left[\begin{array}{ccccc}
O & \A_{12} & O & \cdots & O\\
O & O & \A_{23} & \cdots & O \\
\vdots & \vdots & \vdots & \ddots & \vdots\\
O & O & O & \cdots  & \A_{k-1,1}\\
\A_{k1} & O & O & \cdots & O
\end{array}\right],
\]
where the diagonal blocks are all square zero matrices of suitable sizes.
Equivalently, we have a partition $[n]=V_1 \cup \cdots \cup V_k$ such that if $j \not\equiv i+1 \mod k$, then
\begin{equation} \label{mat-stru1}
\A[V_i|V_j]=0.
\end{equation}

The Perron-Frobenius theorem for nonnegative tensor is established by Chang et.al \cite{CPZ}, Friedland et.al \cite{FGH} and Yang et.al \cite{YY,YY2,YY3}.
From those work, the spectral symmetry of an irreducible or weakly irreducible nonnegative tensor $\A$ is investigated.
The eigenvalues of $\A$ with modulus $\rho(\A)$ are also equally distributed on the spectral circle.
However, the structure of $\A$ receives little attention.
Does $\A$ has a similar structure to that in (\ref{mat-stru1})?
The key difference is that an irreducible nonnegative matrix has a unique (positive) eigenvector corresponding to the spectral radius up to a scalar (called {\it Perron vector}), but an irreducible or weakly irreducible nonnegative tensor may have more than one eigenvector corresponding to the spectral radius up to a scalar.

In order to obtain the structural information of weakly irreducible nonnegative tensors, we start from the discussion of spectral symmetry of tensors.

\begin{definition}\label{ell-sym}
Let $\A$ be an $m$-th order $n$-dimensional tensor, and let $\ell$ be a positive integer.
The tensor $\A$ is called spectral $\ell$-symmetric if
\begin{equation}\label{sym-For}\S(\A)=e^{\i \frac{2\pi}{\ell}}\S(\A).\end{equation}
\end{definition}

Suppose that $\A$ is spectral $\ell$-symmetric.
The maximum number $\ell$ such that (\ref{sym-For}) holds is called the {\it cyclic index} of $\A$ and denoted by $c(\A)$ \cite{CPZ2}, and $\A$ is called
{\it spectral $c(\A)$-cyclic}.
Obviously, if $\A$ is spectral $k$-cyclic, it is spectral $k$-symmetric; and for any positive integer $\ell$ such that $\ell | k$, it is also spectral $\ell$-symmetric.
In particular, if $\A$ is spectral $2$-symmetric, then $\la$ is an eigenvalue of $\A$ if and only $-\la$ is an eigenvalue of $\A$;
in this case, we say $\A$ has a {\it symmetric spectrum}.
If $c(\A)=1$, then $\A$ is spectral $1$-cyclic, and is also said {\it spectral nonsymmetric}.

If a nonnegative tensor $\A$ holds one of the following properties: (1) $\A$ is positive \cite{YY}; (2) $\A$ is primitive \cite{CPZ2}; (3) $A$ is irreducible \cite{YY2} or weakly irreducible \cite{YY3} with positive trace, then $\A$ is spectral nonsymmetric.
Nikiforov \cite{Ni} characterize a symmetric weakly irreducible nonnegative tensor with a symmetric spectrum by introducing the odd-coloring of a tensor,
  where an {\it odd-coloring} is exactly an $(m,2)$-coloring in the following definition for $m$ being even.

\begin{definition}\label{spe-ell-sym}
Let $m \ge 2$ and $\ell \ge 2$ be integers such that $ \ell \mid  m$.
An $m$-th order $n$-dimensional tensor $\A$ is called $(m,\ell)$-colorable
if there exists a map $\phi:[n] \to [m]$ such that if $a_{i_1\ldots i_m} \ne 0$, then
\begin{equation}\label{gen-col} \phi(i_1)+\cdots+\phi(i_m) \equiv \frac{m}{\ell} \mod m.\end{equation}
Such $\phi$ is called an $(m,\ell)$-coloring of $\A$.
\end{definition}

By the results in \cite{YY2,YY3}, for a weakly irreducible nonnegative tensor $\A$ of order $m$, if $\A$ is spectral $\ell$-symmetric, then there exists a diagonal matrix $D$ such that $\A=e^{-\i \frac{2\pi}{\ell}}D^{-(m-1)}\A D$, where $D$ is constructed from an eigenvector corresponding to the eigenvalue $\rho(A)e^{\i \frac{2\pi}{\ell}}$.
Let
\begin{equation} \label{Dg}
\begin{split}
\Dg &=\cup_{j=0}^{\ell-1}\Dg^{(j)}, \\
\Dg^{(j)}& =\{D: \A=e^{-\i \frac{2\pi j}{\ell}}D^{-(m-1)}\A D, d_{11}=1\}, j=0,1,\ldots,\ell-1.
\end{split}
\end{equation}
In the case of $\A$ being a matrix, for $j=0,1,\ldots,\ell-1$, each $ \Dg^{(j)}$ contains only one element,
and $\Dg$ is a group of order $\ell$ under the usual matrix multiplication.
However, in the general case of $\A$ being a tensor, each $ \Dg^{(j)}$ may have more than one element, and contains more rich content.
We show that $\Dg$ is a finite abelian group containing $\Dg^{(0)}$ as a subgroup.
Let $S=\{e^{\i \frac{2 \pi i}{\ell}}\A: i=0,1,\ldots,\ell-1\}$.
Then $\Dg$ acts on $S$ as a permutation group,
where $\Dg^{(0)}$ acts as a stabilizer of $\A$, and the quotient group $\Dg/\Dg^{(0)}$ acts as a rotation over $S$.
The spectral symmetry of $\A$ is characterized by $\Dg/\Dg^{(0)}$.

In this paper we mainly investigate the structure of a weakly irreducible nonnegative tensor $\A$ by the group $\Dg$.
The paper divides into two parts.
First we give some properties of $\Dg$ defined in (\ref{Dg}), and then obtain the structural information of $\A$ similar to (\ref{mat-stru1})
   with application to the edge distribution and coloring of connected hypergraphs.
Consequently, we prove that a symmetric weakly irreducible nonnegative tensor is spectral $\ell$-symmetric
  if and only if it is $(m,\ell)$-colorable, which generalizes the result of Nikiforov \cite{Ni}.
In the second part, we characterize the spectral $\ell$-symmetry and the cyclic index of a tensor by its generalized traces,
  which generalizes the result of Shao et.al \cite{SQH}.
We also prove that for an arbitrarily given integer $m \ge 3$ and for each positive integer $\ell$ such that $\ell \mid m$,
there always exists an $m$-uniform hypergraph $G$ such that its adjacency tensor is spectral $\ell$-symmetric.

\section{Preliminaries}
\subsection{Notions}
Let $\A$ be a real tensor of order $m$ and dimension $n$.
The tensor $\A$ is called \textit{symmetric} if its entries are invariant under any permutation of their indices.
 Given a vector $x\in \mathbb{R}^{n}$, $\A x^{m} \in \mathbb{R}$ and $\A x^{m-1} \in \mathbb{R}^n$, which are defined as follows:
   \[
   \begin{split} \A x^{m}&=\sum_{i_1,i_{2},\ldots,i_{m}\in [n]}a_{i_1i_{2}\ldots i_{m}}x_{i_1}x_{i_{2}}\cdots x_{i_m},\\
   (\A x^{m-1})_i &=\sum_{i_{2},\ldots,i_{m}\in [n]}a_{ii_{2}\ldots i_{m}}x_{i_{2}}\cdots x_{i_m}, i \in [n].
   \end{split}
   \]
 Let $\mathcal{I}=(i_{i_1i_2\ldots i_m})$ be the {\it identity tensor} of order $m$ and dimension $n$, that is, $i_{i_{1}i_2 \ldots i_{m}}=1$ if and only if
   $i_{1}=i_2=\cdots=i_{m} \in [n]$ and $i_{i_{1}i_2 \ldots i_{m}}=0$ otherwise.

\begin{definition}{\em \cite{Qi2}} Let $\A$ be an $m$-th order $n$-dimensional real tensor.
For some $\lambda \in \mathbb{C}$, if the polynomial system $(\lambda \mathcal{I}-\A)x^{m-1}=0$, or equivalently $\A x^{m-1}=\lambda x^{[m-1]}$, has a solution $x\in \mathbb{C}^{n}\backslash \{0\}$,
then $\lambda $ is called an eigenvalue of $\A$ and $x$ is an eigenvector of $\A$ associated with $\lambda$,
where $x^{[m-1]}:=(x_1^{m-1}, x_2^{m-1},\ldots,x_n^{m-1})$.
\end{definition}

If $x$ is a real eigenvector of $\A$, surely the corresponding eigenvalue $\lambda$ is real.
In this case, $\lambda$ is called an {\it $H$-eigenvalue} of $\A$.
The {\it spectral radius of $\A$} is defined as
\[
\rho(\A)=\max\{|\lambda|: \lambda \mbox{ is an eigenvalue of } \A \}.
\]

\begin{definition}[\cite{CPZ}]
A tensor $\A=(a_{i_{1}i_2\ldots i_{m}})$ of order $m$ and dimension $n$ is called reducible if there exists a nonempty proper index subset $I \subset [n]$ such that
$a_{i_{1}i_2\ldots i_{m}}=0$ for any $i_1 \in I$ and any $i_2,\ldots,i_m \notin I$. If $\A$ is not reducible, then it is called irreducible.
\end{definition}

For a tensor $\A$ of order $m$ and dimension $n$, we associate it with a directed graph $D(\A)$ on vertex set $[n]$ such that $(i,j)$ is an arc of $D(\A)$ if
and only if there exists a nonzero entry $ a_{ii_2\ldots i_{m}}$ such that $j \in \{ i_2\ldots i_{m}\}$.
The tensor $\A$ is called {\it weakly irreducible} if $D(\A)$ is strongly connected; otherwise it is called {\it weakly reducible} \cite{FGH}.
It is known that if $\A$ is irreducible, then it is weakly irreducible; but the converse is not true.

A {\it hypergraph} $G=(V(G),E(G))$ consists of a vertex set  $V(G)=\{v_1,v_2,\ldots,v_n\}$ and an edge set $E(G)=\{e_{1},e_2,\ldots,e_{l}\}$
  where $e_{j}\subseteq V(G)$ for $j \in [l]$.
If $|e_{j}|=m$ for each $j\in [l]$, then $G$ is called an {\it $m$-uniform} hypergraph.
In particular, the $2$-uniform hypergraphs are exactly the classical simple graphs.
The {\it degree} $d_G(v)$ or simply $d(v)$ of a vertex $v \in V(G)$ is defined as $d(v)=|\{e_{j}:v\in e_{j}\in E(G)\}|$.
A {\it walk} $W$ in $G$ is a sequence of alternate vertices and edges: $v_{0},e_{1},v_{1},e_{2},\ldots,e_{l},v_{l}$,
    where $\{v_{i},v_{i+1}\}\subseteq e_{i+1}$ for $i=0,1,\ldots,l-1$.
The hypergraph $G$  is {\it connected} if every two vertices of $G$ are connected by a walk,
  and is called {\it $k$-colorable} if there exists a map $\phi: V(G) \to [k]$ such that each edge contains at least two vertices with different colors,
or equivalently, the vertices can be partitioned into $k$ subsets such that each edge intersects at least two subsets.
The {\it chromatic number} $\chi(G)$ is the smallest $k$ such that $G$ is $k$-colorable.
The {\it adjacency tensor} $\A(G)$ of the hypergraph $G$ is defined as $\mathcal{A}(G)=(a_{i_{1}i_{2}\ldots i_{m}})$, an $m$-th order $n$-dimensional tensor, where
\[a_{i_{1}i_{2}\ldots i_{m}}=\left\{
 \begin{array}{ll}
\frac{1}{(m-1)!}, &  \mbox{if~} \{v_{i_{1}},v_{i_{2}},\ldots,v_{i_{m}}\} \in E(G);\\
  0, & \mbox{otherwise}.
  \end{array}\right.
\]
  In this paper, the eigenvalues of a hypergraph $G$ always refer to those of its adjacency tensor.
 Let $\mathcal{D}(G)$ be an $m$-th order $n$-dimensional diagonal tensor,
    where $d_{i\ldots i}=d(v_i)$ for $i \in [n]$.
The tensor $\L(G)=\mathcal{D}(G)-\A(G)$ is called the {\it Laplacian tensor} of $G$,
and $\Q(G)=\mathcal{D}(G)+\A(G)$ is called the {\it signless Laplacian tensor} of $G$.

Observe that the adjacency (Laplacian, signless Laplacian) tensor of a hypergraph $G$ is symmetric,
  and it is weakly irreducible if and only if the $G$ is connected \cite{PZ,YY3}.
However, even if $G$ is connected, the tensor $\A(G)$ ($\L(G)$, $\Q(G)$) is always reducible when $m \ge 3$;
for taking an arbitrary proper subset $I \subset [n]$ with cardinality not less than $n-(m-2)$, we always have $a_{i_{1}i_{2}\ldots i_{m}}=0$
for all $i_1 \in I$ and all $i_2,\ldots,i_m \notin I$ since there must exist repeated indices among $i_2,\ldots,i_m$.

Let $I_X$ be the indicator function of a set $X \subset [n]$, and let $\A$ be a tensor of order $m$ and dimension $n$.
A set $X \subset [n]$ is called an {\it odd transversal} of $\A$ if $a_{i_1\ldots i_m} \ne 0$ implies that
\[I_X(i_1)+\cdots+I_X(i_m) \equiv 1 \mod 2.\]
A tensor $\A$ with an odd transversal is called an {\it odd transversal tensor} \cite {Ni}.
Odd transversal tensors were also named {\it weakly odd-bipartite tensors} by Chen an Qi \cite{CQ},
 and were called {\it odd-bipartite tensors} in case of $m$ being even.

When we say a hypergraph is {\it spectral $\ell$-symmetric} (or {\it spectral $\ell$-cyclic}, {\it $(m,\ell)$-colorable}, {\it odd-colorable}, {\it odd transversal}),
it always referred to its adjacency tensor.
An even uniform hypergraph is called  {\it odd-bipartite} if its vertices can be partitioned into two subsets
   such that every edge intersects with each subset with an odd number of vertices.

Nikiforov \cite{Ni} proved that if $m$ is even, then an $m$-th order tensor with an odd transversal is always odd-colorable.
Furthermore, if $m \equiv 2 \mod 4$, then these two notions are equivalent.
However, if $m \equiv 0 \mod 4$, they construct two classes of $m$-uniform hypergraphs to illustrate that they are odd-colorable but not odd transversal.
We also construct a class of non-odd-bipartite generalized power hypergraphs to illustrate the above fact \cite{FKT}.

Finally we introduce some special classes of hypergraphs.
An $m$-uniform hypergraph $G$ is called {\it $p$-\hbox{hm} bipartite} if the vertex set  has a bipartition $V(G)=V_1 \cup V_2$ such that
each edge of $G$ intersects $V_1$ with exactly $p$ vertices \cite{SQH}.
The notion of $p$-hm hypergraphs generalizes the hm-bipartite hypergraphs \cite{HQ} (i.e. $1$-hm hypergraphs), and $m$-partite hypergraphs \cite{CD}.
A {\it cored hypergraph} is one such that each edge contains one vertex of degree one \cite{HQS}, which is also hm-bipartite.
An {\it  $m$-th power} of a simple graph $G$, denoted by $G^m$, is obtained from $G$ by replacing each edge (a $2$-set) with a $k$-set by adding $(k-2)$ additional vertices \cite{HQS},
 which is a cored hypergraph and hence hm-bipartite.
A {\it generalized power hypergraph} $G^{m,s}$  is constructed from a simple graph $G$ by Khan and Fan \cite{KF} and from a hypergraph $G$ by Kang et.al \cite{KLQY}.
In particular, if $t |m$, then $G^{m,\frac{m}{t}}$ is simply obtained from $G$ by blowing up each vertex into an $\frac{m}{t}$-set.

\begin{definition}[\cite{KLQY}] \label{powerhyper}
Let $G=(V,E)$ be a $t$-uniform hypergraph. For any integers $m,s$ such that $m > t$ and $1 \le s \le \frac{m}{t}$, the generalized power of $G$, denoted by $G^{m,s}$, is defined as the $m$-uniform hypergraph with the vertex set $(\cup_{v \in V} \v) \cup (\cup_{e \in E}\mathbf{e})$, and the edge set
$\{\u_1 \cup \cdots \cup \u_t \cup \mathbf{e}: e=\{u_1,\ldots,u_t\} \in E(G)\}$,
where $\v$ denotes an $s$-set corresponding to $v$ and $\mathbf{e}$ denotes an $(m-ts)$-set corresponding to $e$, and all those sets are pairwise disjoint.
\end{definition}

\subsection{Characteristic polynomial of tensors}
Let $\A$ be an $m$-th order $n$-dimensional real tensor.
The {\it determinant } of $\A$, denoted by $\det \A$, is defined as the resultant of the polynomials $\A x^{m-1}$,
and the {\it characteristic polynomial } $\varphi_\A(\la)$ of $\A$ is defined as $\det(\la \I-\A)$ \cite{CPZ2,Qi2}.
It is known that $\la$ is an eigenvalue of $\A$ if and only if it is a root of $\varphi_\A(\la)$.
The {\it algebraic multiplicity} of $\la$ as an eigenvalue of $\A$  is defined as the multiplicity of $\la$ as a root of $\varphi_\A(\la)$.
The {\it spectrum} of $\A$, denoted by $\S(\A)$, is the multi-set of the roots of $\varphi_\A(\la)$, including multiplicity.
Denote by $H\S(\A)$ the set of distinct $H$-eigenvalues of $\A$.

Morozov and Shakirov \cite{MS} give a formula for calculating $\det (\I-\A)$ using Schur polynomials in the generalized traces of $\A$.
Let $A$ be an auxiliary matrix of order $n$ with distinct variables $a_{ij}$ as entries.
The {\it generalized $d$-th order trace} of $\A$ is defined by
\[ \label{trace}\Tr_d(\A)=(m-1)^{n-1}\sum_{d_1+\cdots+d_n=d}\prod_{i=1}^n \frac{1}{(d_i(m-1))!}\left(\sum_{y \in [n]^{m-1}}t_{iy}\frac{\partial }{\partial a_{iy}}\right)^{d_i}\Tr(A^{d(m-1)}),\]
where $d_1,\ldots, d_n$ are nonnegative integers,
$t_{iy}=t_{ii_2\ldots i_m}$ and $\frac{\partial }{\partial a_{iy}}=\frac{\partial }{\partial a_{ii_2}} \cdots \frac{\partial }{\partial a_{i i_m}}$
if $y=i_2\ldots i_m$.
By the results in \cite{CD,HHLQ, SQH}
\[\label{charp}\varphi_\A(\la)=\sum_{i=0}^D P_i\left(-\frac{\Tr_1(\A)}{1},-\frac{\Tr_2(\A)}{2},\cdots,-\frac{\Tr_i(\A)}{i}\right)\la^{D-i}, ~~(D=n(m-1)^{n-1}),\]
where the Schur polynomial
\begin{equation}\label{schur}P_d(t_1,\cdots,t_d)=\sum_{m=1}^d \sum_{d_1+\cdots+d_m=d, d_i \in \mathbb{Z}^+} \frac{t_{d_1} \cdots t_{d_m}}{m!}.\end{equation}
It was proved in \cite{HHLQ, SQH} that
\begin{equation} \label{traceF}\Tr_d(\A)=\sum_{i=1}^D \la_i^d,\end{equation}
where $\la_1,\ldots,\la_D$ are all eigenvalues of $\A$.

Shao, Qi and Hu \cite{SQH} give a graph theoretic formula for the trace $\Tr_d(\A)$.
Denote
\[\label{Fd}\F_d=\{((i_1,\alpha_1),\cdots, (i_d,\alpha_d)):1 \le i_1 \le \cdots i_d \le n; \alpha_1, \cdots, \alpha_n \in [n]^{m-1}\}.\]
For each $F=((i_1,\alpha_1),\cdots, (i_d,\alpha_d)) \in \F_d$, define a directed graph $D(F)$ with arc multi-set
$E(F)=\cup_{j=1}^d E_j(F)$, where $E_j(F)=\{(i_j,v_1),(i_j,v_2),\ldots,(i_j,v_{m-1})\}$ if $\alpha_j=(v_1,\ldots,v_{m-1})$.
Here, for each tuple $(i_j,\alpha_j)$, $i_j$ is called the {\it primary index} and the indices in $\alpha_j$ are called {\it secondary indices}.

In the directed graph $D(F)$, denote by $b(F)$ the product of the factorials of the multiplicities of all arcs of $E(F)$,
  by $c(F)$ the product of the factorials of the out-degrees of  all vertices incident to the arcs of $E(F)$,
  and by $\W(F)$ the set of all closed walks with the arc  multi-set $E(F)$.
Denote by $\Pi_F(\A)=\prod_{j=1}^d t_{i_j,\alpha_j}$.
Then
\begin{equation}\label{tracedF}\Tr_d(\A)=(m-1)^{n-1}\sum_{F \in \F_d}\frac{b(F)}{c(F)}\Pi_F(\A)|\W(F)|.\end{equation}

If one summand of (\ref{tracedF}) is nonzero for some $F=((i_1,\alpha_1),\cdots, (i_d,\alpha_d))  \in \F_d$, then
$\Pi_F(\A)=\prod_{j=1}^d t_{i_j,\alpha_j}$ is {\it $m$-valent}, i.e. each index occurs in the monomial
$\Pi_F(\A)$ in times of multiple of $m$ \cite{CD,SQH};
furthermore, the directed graph $D(F)$ contains a Eulerian directed circuit, i.e. $\W(F) \ne \emptyset$, or equivalently
$D(F)$ is connected and each vertex of $D(F)$ has the same in-degree and out-degree.
If, in addition, $\A=\A(G)$ for some hypergraph $G$, then, by omitting the order of tuples in $F$,  the above $F$ can be written as a set
\begin{equation}\label{F}F =\{e_1(i_1),\ldots,e_d(i_d)\},\end{equation}
where $e_j(i_j)$ denotes an edge $e_j$ of $G$ with primary index (vertex) $i_j$ for $j \in [d]$.
If we write $\varphi_\A(\la)=\sum_{i=0}^D a_i\la^{D-i}$, where $a_0=1$, then
\begin{equation}\label{coef}a_i=P_i\left(-\frac{\Tr_1(\A)}{1},-\frac{\Tr_2(\A)}{2},\cdots,-\frac{\Tr_i(\A)}{i}\right), i=1,\ldots,D.\end{equation}

\subsection{Perron-Frobenius theorem for nonnegative tensors}
Chang et.al \cite{CPZ} generalize the Perron-Frobenius theorem for nonnegative matrices to nonnegative tensors.
Yang and Yang \cite{YY,YY2,YY3} get some further results for Perron-Frobenius theorem, especially for the spectral symmetry.
Friedland et.al \cite{FGH} also get some results for weakly irreducible nonnegative tensors.
We combine those results in the following theorem, where an eigenvalue is called {\it $H^+$-eigenvalue} (respectively {\it $H^{++}$-eigenvalue}) if it is associated with
a nonnegative (respectively positive) eigenvector.

\begin{theorem}[The Perron-Frobenius Theorem for nonnegative tensors]\label{PF1}~~

\begin{enumerate}

\item{\em(Yang and Yang \cite{YY})}  If $\A$ is a nonnegative tensor of order $m$ and dimension $n$, then $\rho(\A)$ is an $H^+$-eigenvalue of $\A$.

\item{\em(Friedland, Gaubert and Han \cite{FGH})} If furthermore $\A$ is weakly irreducible, then $\rho(\A)$ is the unique $H^{++}$-eigenvalue of $\A$,
with the unique positive eigenvector, up to a positive scalar.

\item{\em(Chang, Pearson and Zhang \cite{CPZ})} If moreover $\A$ is irreducible, then $\rho(\A)$ is the unique $H^{+}$-eigenvalue of $\A$,
with the unique nonnegative eigenvector, up to a positive scalar.

\end{enumerate}

\end{theorem}

According to the definition of tensor product in \cite{Shao}, for a tensor $\A$ of order $m$ and dimension $n$, and two diagonal matrices $P,Q$ both of dimension $n$,
the product $P\A Q$ has the same order and dimension as $\A$, whose entries are defined by
\[(P\A Q)_{i_1i_2\ldots i_m}=p_{i_1i_1}a_{i_1i_2\ldots i_m}q_{i_2i_2}\ldots q_{i_mi_m}.\]
If $P=Q^{-1}$, then $\A$ and $P^{m-1}\A Q$ are called {\it diagonal similar}.
It is proved that two diagonal similar tensors have the same spectrum \cite{Shao}.

\begin{theorem}[\cite{YY3}] \label{PF2}
 Let $\A$ and $\mathcal{B}$ be two $m$-th order $n$-dimensional real tensors with $|\mathcal{B}| \le  \A$. Then

\begin{enumerate}

\item $\rho(\mathcal{B}) \le \rho(\A)$.

\item Furthermore, if $\A$ is weakly irreducible and $\rho(\mathcal{B}) = \rho(\A)$, where $\la=\rho(A)e^{\i \theta}$ is an eigenvalue of $\mathcal{B}$ corresponding to an
eigenvector $y$,  then $y$ contains no zero entries, and $\mathcal{B}=e^{-\i \theta}D^{-(m-1)}\A D$, where $D=\hbox{diag}\{\frac{y_1}{|y_1|},\ldots,\frac{y_n}{|y_n|}\}$.
\end{enumerate}
\end{theorem}

\begin{theorem}[\cite{YY3}] \label{PF3}
Let $\A$ be an $m$-th order $n$-dimensional weakly irreducible nonnegative tensor.
Suppose $\A$ has $k$ distinct eigenvalues with modulus $\rho(\A)$ in total.
Then these eigenvalues are $\rho(\A)e^{\i \frac{2 \pi j}{k}}$, $j=0,1,\ldots,k-1$.
Furthermore, \begin{equation}\A =e^{-\i \frac{2\pi}{k}}D^{-(m-1)}\A D,\end{equation}
and the spectrum of $\A$ keeps invariant under a rotation of angle $\frac{2\pi}{k}$ (but not a smaller positive angle) of the complex plane.
\end{theorem}

Suppose $\A$ be as in Theorem \ref{PF3}.
By Theorem \ref{PF2} and Theorem \ref{PF3}, if $\S(\A)$ is invariant under a rotation of angle $\theta$ of the complex plane,
then $\theta=\frac{2 \pi j}{k}$ for some positive $k$ and some $j \in [k]$.
So   \begin{equation} \S(\A)=e^{\i \frac{2\pi}{k}}\S(\A).\end{equation}
This is the motivation of our Definition \ref{ell-sym}.
The number $k$ in Theorem \ref{PF3} is exactly the cyclic index of $\A$.
In addition, if $\A$ is spectral $\ell$-symmetric,
Then $\ell \mid  c(\A)$ by Theorem \ref{PF3}.
We have a more generalized result as follows.

\begin{lemma}
Let $\A$ be an $m$-th order $n$-dimensional tensor.
If $\A$ is spectral $\ell$-symmetric, then $\ell \mid  c(\A)$.
\end{lemma}

\begin{proof}  Let $c:=c(\A)$.
Assume to the contrary, $\ell \nmid  c$.
Then there exists an integer $h$ such that $\frac{2\pi h}{c} <  \frac{2\pi}{\ell} <\frac{2\pi(h+1)}{c}$,
and hence $\theta:=\frac{2\pi}{\ell}-\frac{2\pi h}{c}<\frac{2\pi }{c}$.
Write $\theta=\frac{2\pi q }{p}$, where $(p,q)=1$ and $p>c$.
As $\A$ is both spectral $\ell$-symmetric and spectral $c$-symmetric, we have
\[\S(\A)=e^{\i \theta}\S(\A)=e^{\i \frac{2\pi q }{p}}\S(\A).\]
Since $(p,q)=1$, there exist integers $h_1,h_2$ such that $ph_1+qh_2=1$.
So \[\S(\A)=e^{\i \frac{2\pi q }{p}}\S(\A)=e^{\i \frac{2\pi qh_2 }{p}}\S(\A)=e^{\i \frac{2\pi }{p}}\S(\A),\]
which implies that $\A$ is spectral $p$-symmetric, a contradiction as $p>c=c(\A)$.
\end{proof}

\section{Structure of nonnegative weakly irreducible tensors}
In this section we will first analyze the property of the group $\Dg$ defined in (\ref{Dg}) by the theory of finite abelian group.
Then we obtain some structural information of a  weakly irreducible nonnegative tensor, with application to the edge distribution and coloring of connected hypergraphs.

\begin{lemma}[\cite{YY3}] \label{ev}
Let $\A$ be an $m$-th order $n$-dimensional weakly irreducible nonnegative tensor.
Let $y$ be an eigenvector of $\A$ corresponding to an eigenvalue $\la$ with $|\la|=\rho(\A)$.
Then $|y|$ is the unique positive eigenvector corresponding to $\rho(\A)$ up to a scalar.
\end{lemma}

Let $r$ be a positive integer and let $p$ be a prime number.
Denote by $r_{[p]}$ be the maximum power of $p$ that divides $r$.
Surely, if $p \nmid r$, then $r_{[p]}=1$.

\begin{lemma} \label{group}
Let $\A$ be an $m$-th order $n$-dimensional weakly irreducible nonnegative tensor, which is spectral $\ell$-symmetric.
Suppose $\A$ has $r$ distinct eigenvectors corresponding to $\rho(\A)$ in total up to a scalar.
Let $\Dg$ and $\Dg^{(j)}$ be as defined in (\ref{Dg}) for $j=0,1,\ldots,\ell-1$.
Then the following results hold.

\begin{enumerate}

\item $\Dg$ is a finite abelian group of order $r\ell$ under  matrix multiplication, where $\Dg^{(0)}$ is a subgroup of $\Dg$ of order $r$, and
$\Dg^{(j)}$ is a coset of $\Dg^{(0)}$ in $\Dg$ for $j=1,\ldots,\ell-1$.

\item For each prime factor $p$ of $\ell$, there exists a matrix $D \in \Dg \backslash \Dg^{(0)}$ such that $D^{r_{[p]} p}=I$ and
hence $D^{r_{[p]} \ell}=D^{rp}=D^{r \ell}=I$.

\item If further $\A$ is symmetric, then $\ell\mid m$, and $D^{m}=I$ for any $D \in \Dg$.
In particular, each elementary divisor of $\Dg$ divides $m$ and each prime factor of $r \ell$ divides $m$.
\end{enumerate}

\end{lemma}

\begin{proof} (1) Surely the identity matrix $I \in \Dg^{(0)}$. For any two matrices $D^{(j_1)} \in \Dg^{(j_1)}$ and $D^{(j_2)} \in \Dg^{(j_2)}$, we have
\[\A = e^{-\i \frac{2\pi j_1}{\ell}} {D^{(j_1)}}^{-(m-1)}\A D^{(j_1)},
    \A=e^{-\i \frac{2\pi j_2}{\ell}} {D^{(j_2)}}^{-(m-1)}\A D^{(j_2)}.\]
Then \[ \A = e^{-\i \frac{2\pi (j_1+j_2)}{\ell}} {(D^{(j_1)}D^{(j_2)})}^{-(m-1)}\A (D^{(j_1)}D^{(j_2)}).\]
So $D^{(j_1)}D^{(j_2)} \in \Dg^{(j_1+j_2)}$ where the superscript is taken modulo $\ell$.
It is seen that $(D^{(j_1)})^{-1}\in \Dg^{(-j_1)}$.
So $\Dg$ is a group under the usual  matrix multiplication.

Following the same routine, one can verify that $\Dg^{(0)}$ is a subgroup of $\Dg$.
Taking a $D^{(j)} \in \Dg^{(j)}$, one can show that $\Dg^{(j)}=\Dg^{(0)} D^{(j)}$, i.e. $\Dg^{(j)}$ is a coset of $\Dg^{(0)}$
by verifying $\bar{D}^{(j)}{D^{(j)}}^{-1} \in \Dg^{(0)}$ and $D^{(0)}{D^{(j)}} \in \Dg^{(j)}$ for any $\bar{D}^{(j)} \in \Dg^{(j)}$ and $D^{(0)} \in \Dg^{(0)}$.

Next we will show $\Dg^{(0)}$ has $r$ elements, and hence $\Dg$ has $r\ell$ elements by the above discussion.
Let $y^{(01)},\ldots,y^{(0r)}$ be the $r$ distinct eigenvectors of $\A$ corresponding to $\rho(\A)$ up to a scalar, each of which contains no zero entries by Lemma \ref{ev}.
Without loss of generality, assume $y^{(0j)}_1=1$ for $j \in [r]$.
Define \begin{equation} \label{diag} D^{(0j)}=\hbox{diag}\left(\frac{y^{(0j)}_1}{\left|y^{(0j)}_1\right|},\ldots,\frac{y^{(0j)}_n}{\left|y^{(0j)}_n\right|}\right), j=1,\ldots,r.\end{equation}
By Theorem \ref{PF2}, $D^{(0j)} \in \Dg^{(0)}$ for $j \in [r]$.
By Theorem \ref{PF1}(2), we may assume $y^{(01)}>0$ so that $D^{(01)}=I$.
Now suppose that $D \in \Dg^{(0)}$.
From the equalities
\[ \A (y^{(01)})^{m-1}=\rho(\A) (y^{(01)})^{[m-1]}, \; \A = D^{-(m-1)}\A D,\]
we get
\[ \label{ev-diag2}\A (Dy^{(01)})^{m-1}=\rho(\A)(Dy^{(01)})^{[m-1]}.\]
So $Dy^{(01)}$ is an eigenvector of $\A$ corresponding to $\rho(\A)$ with $(Dy^{(01)})_1=1$,
  which implies that $Dy^{(01)}=y^{(0l)}$ for some $l \in [r]$ and $D=D^{(0l)}$.

(2) Let $p$ be a prime factor of $\ell$.
 First assume $r_{[p]} =1$, i.e. $p \nmid r$.
 By Cauchy Theorem, $\Dg$ contains an element $D$ of order $p$.
Since $p \nmid r$ and $\Dg^{(0)}$ has order $r$, $D \notin \Dg^{(0)}$ by Lagrangian Theorem.

Next suppose that $ p \mid r$.
Let $ \Dg(p),\Dg^{(0)}(p)$ be the Sylow $p$-subgroups of $\Dg$ and $\Dg^{(0)}$ respectively.
Observe that $\Dg^{(0)}(p)$ is a proper subgroup of $\Dg(p)$, and
\[
\Dg(p) \cong \mathbb{Z}_{p^{e_{1}}} \oplus \cdots \oplus \mathbb{Z}_{p^{e_{\tau}}},
\Dg^{(0)}(p) \cong \mathbb{Z}_{p^{f_{1}}} \oplus \cdots \oplus \mathbb{Z}_{p^{f_{\xi}}},
\]
where $e_{1} \ge \cdots \ge e_{\tau} \ge 1$ and $f_{1} \ge \cdots \ge f_{\xi} \ge 1$.
Note that  $e_{1} \ge f_{1}$.
If $e_{1} > f_{1}$, then $\Dg \backslash \Dg^{(0)}$ contains an element $D$ of order $p^{e_{1}}$.
So $D^{p^{e_{1}}}=I$, and hence $(D^{p^{e_{1}-f_{1}-1}})^{p^{f_{1}+1}}=I$.
Let $\hat{D}=D^{p^{e_{1}-f_{1}-1}}$.
Then $\hat{D}$ has order $p^{f_{1}+1}$ that divides $r_{[p]} p$, implying that $\hat{D} \notin \Dg^{(0)}$, as desired.
Otherwise, $e_{1} = f_{1}$.
We consider $\mathbb{Z}_{p^{e_{2}}}$ and $\mathbb{Z}_{p^{f_{2}}}$, and repeat the above process.
Finally we have two cases: (i) there exists a $j \in [\tau]$ such that  $e_{j}> f_{j}$,
(ii) $\tau \ge \xi+1$, and for each $j \in [\xi]$, $e_{j} = f_{j}$.
If the case (i) occurs, we will obtain a desired matrix like $\hat{D}$ as in the above.
Otherwise, there is an element $D  \in \Dg \backslash \Dg^{(0)}$ of order $p^{e_{\xi+1}}$, implying that
$\Dg \backslash \Dg^{(0)}$ contains an element of order $p$, as desired.

(3) Suppose $\A$ is symmetric.
From the equality $ \A =e^{-\i \frac{2\pi j}{\ell}}  D^{-(m-1)}\A D$, letting $d_{ii}=e^{\i \theta_i}$ for $i \in [n]$ where $\theta_1=0$,
  if $a_{i_1 \ldots i_m} \ne 0$, we have
 \begin{equation}\label{ell-j-eq}\frac{2\pi j}{\ell}+m\theta_{i_1} \equiv \theta_{i_1}+\cdots+\theta_{i_m} \mod 2\pi.\end{equation}
  As $\A$ is symmetric, replacing $i_1$ in left side of (\ref{ell-j-eq}) by $i_l$ and summing over all $l=1,\ldots,m$,
 \[\frac{2\pi jm}{\ell}+m\sum_{l=1}^m\theta_{i_l}\equiv m\sum_{l=1}^m\theta_{i_l} \mod 2\pi.\]
  So we have $\frac{2\pi jm}{\ell}\equiv 0 \mod 2\pi$.
  If taking $j=1$, we have $\ell\mid m$.
  Also from (\ref{ell-j-eq}), we have
\[m\theta_{i_1} = \ldots=m\theta_{i_m} \mod 2\pi.\]
  As $\A$ is weakly irreducible and $\theta_1=0$,
  we get $m\theta_i \equiv 0 \mod 2\pi.$
  So $D^m=I$ for any $D \in \Dg$.    The result follows.
  \end{proof}

\begin{remark}  Under the assumption of Theorem \ref{group}, if taking $\ell=c(\A)$, then all the results also hold.
In particular, if $\A$ is symmetric, then $c(\A)\mid m$, which is proved in \cite{YY2}.
Let $S=\{e^{\i \frac{2 \pi i}{\ell}}\A: i=0,1,\ldots,\ell-1\}$.
Then $\Dg$ acts on $S$ as a permutation group,
where $\Dg^{(0)}$ acts as a stabilizer of $\A$, and the quotient $\Dg/\Dg^{(0)}$ acts as a rotation over $S$.
The spectral symmetry of $\A$ is characterized by $\Dg/\Dg^{(0)}$, and $\A$ is spectral $|\Dg/\Dg^{(0)}|$-symmetric.
\end{remark}

Denote $s(\A):=|\Dg^{(0)}|$ and $s(G):=s(\A(G))$ for a uniform hypergraph $G$, which are exactly
the number of distinct eigenvectors of $\A$ and $\A(G)$ corresponding to their spectral radii up to a scalar, respectively.

\begin{coro}\label{m-ev}
Let $\A$ be an $m$-th order $n$-dimensional weakly irreducible nonnegative tensor with $s(\A)=r$,
which is spectral $\ell$-symmetric.
 Let $\la_j=\rho(\A)e^{\i \frac{2\pi j}{\ell}}$ be an eigenvalue of $\A$ corresponding to an eigenvector $y^{(j)}$ with $y^{(j)}_1=1$ for $j =0,1,\ldots,\ell-1$.

\begin{enumerate}

\item If $j=0$, then $y^{[r]}>0$.

\item For each $j \in [\ell-1]$, $(y^{(j)})^{[r \ell]}>0$.

\item For each prime factor $p$ of $\ell$, there exists some $j \in [\ell-1]$ such that $(y^{(j)})^{[r_{[p]} p]}>0$.

\item If further $\A$ is symmetric, then for each $j=0,1,\ldots,\ell-1$, $(y^{(j)})^{[m]}>0$.

\end{enumerate}
\end{coro}

\begin{proof}
Similar to (\ref{diag}), we can construct a diagonal matrix $D$ by the eigenvector $y^{(j)}$:
\[D=\hbox{diag}\left(\frac{y^{(j)}_1}{\left|y^{(j)}_1\right|},\ldots,\frac{y^{(j)}_n}{\left|y^{(j)}_n\right|}\right). \]
By the results in \cite{YY,YY3}, we have $\A=e^{-\i \frac{2\pi j}{\ell}} D^{-(m-1)}\A D$.
If $j=0$, then $D \in \Dg^{(0)}$ and $D^r=I$ by Lemma \ref{group}(1), yielding the result (1).
For each $j \in [\ell-1]$, as $\Dg$ has order $r \ell$, we have $D^{r \ell}=I$, implying the result (2).
The results (3) and (4) can be obtained similarly by Lemma \ref{group}(2-3).
\end{proof}

\begin{lemma}\label{j-sigma}
Let $\A$ be an $m$-th order $n$-dimensional weakly irreducible nonnegative tensor, which is spectral $\ell$-symmetric.
If there exists a diagonal matrix $D \in \Dg^{(j)}$ such that $D \ne I$ and $D^{\sigma}=I$, then
there exists a partition of $[n]=V_1\cup \cdots \cup V_s$ for some integer $s \ge 2$ and a map $\phi: [n] \to [\sigma]$ satisfying
$\phi|_{V_i}=l_i$ for $i \in [s]$, where $l_1(=\sigma),l_2, \ldots, l_s$ are distinct integers, such that
if $a_{i_1\ldots i_m} \ne 0$, then
\begin{equation}\label{j-sigma1}
\frac{j}{\ell}+\frac{m \phi(i_1)}{\sigma} \equiv \frac{\phi(i_1)+\cdots+\phi(i_m)}{\sigma} \mod \mathbb{Z};
\end{equation}
furthermore, if $\A$ is also symmetric, then  $\sigma \mid ml_i$ for each $i \in [s]$, and
\begin{equation}\label{j-sigma2}
\frac{j}{\ell} \equiv \frac{\phi(i_1)+\cdots+\phi(i_m)}{\sigma} \mod \mathbb{Z}.
\end{equation}
\end{lemma}

\begin{proof}
By the definition of $\Dg^{(j)}$, as $D^\sigma=I$,
without loss of generality, we write
\[D=I_{n_1} \oplus e^{\i \frac{2 \pi l_2}{\sigma}}I_{n_2} \oplus \cdots  e^{\i \frac{2 \pi l_s}{\sigma}}I_{n_s},\]
where $s \ge 2$ as $D \ne I$, $I_t$ denotes an identity matrix of dimension $t$, and $l_1(=\sigma), l_2, \ldots, l_s$ are distinct integers not greater than $\sigma$.
So we have a partition of $[n]=V_1 \cup \cdots \cup V_s$ such that $V_i$ consists of the indices indexed by $I_{n_i}$ for $i \in [s]$, and a map $\phi: [n] \to [\sigma]$ satisfying $\phi|_{V_i}=l_i$ for $i \in [s]$.
As $\A=e^{-\i \frac{2 \pi j}{\ell}}D^{-(m-1)}\A D$, if $a_{i_1 \ldots i_m} \ne 0$ for $i_1 \in V_{i_1}, \ldots, i_m \in V_{i_m}$, then
\begin{equation} \label{j-sigma11} \frac{j}{\ell}+\frac{m l_{i_1}}{\sigma} \equiv \frac{l_{i_1}+\cdots+l_{i_m}}{\sigma} \mod \mathbb{Z}. \end{equation}
yielding (\ref{j-sigma1}).

If $\A$ is symmetric, replacing $l_{i_1}$ in the left side of (\ref{j-sigma11}) by $l_{i_j}$ for $j \in [m]$, then (\ref{j-sigma11}) also holds.
So for $j \in [m]$,
\[ \frac{m l_{i_1}}{\sigma} \equiv \frac{m l_{i_j}}{\sigma} \mod \mathbb{Z}. \]
As $\A$ is weakly irreducible and $l_1=\sigma$,
we have $\frac{m l_{i_j}}{\sigma} \equiv 0 \mod \mathbb{Z}$, which implies that $\sigma \mid  m l_j$ for each $j \in [s]$.
So, from (\ref{j-sigma11}) we get
\[ \frac{j}{\ell} \equiv \frac{l_{i_1}+\cdots+l_{i_m}}{\sigma} \mod \mathbb{Z},
\]
yielding (\ref{j-sigma2}).
\end{proof}

Now we will arrive at a result on hypergraph coloring by using spectral symmetry.
By Lemma \ref{group}(1), for any $D \in \Dg$, $D^{r \ell}=I$ as $\Dg$ has order $r \ell$.
But, for the coloring problem, we need as few colors as possible.
So we will use the matrix $D$ in (2) or (3) of Lemma \ref{group}.

\begin{coro}\label{coloring}
Let $G$ be a connected $m$-uniform hypergraph with $s(G)=r$, which is spectral $\ell$-symmetric $(\ell \ge 2)$.
Then the following results hold.

\begin{enumerate}

\item $G$ is $r_{[p]}p$-colorable for each prime number $p \mid \ell$.

\item $G$ is $(m,\ell)$-colorable and $m$-colorable.

\item  $\chi(G) \le \min\{r_{[p]}p,m\}$, where the minimum is taken over all prime numbers $p$ with $p \mid \ell$.

\end{enumerate}
\end{coro}

\begin{proof}
By Lemma \ref{group}(2), there exists a diagonal matrix $D  \in \Dg^{(j)}$ for some $j \in [\ell-1]$ such
that $D^{r_{[p]} p}=I$.
So by Lemma \ref{j-sigma} and (\ref{j-sigma2}), there exists a map $\phi: [n] \to [r_{[p]} p]$ such that
if $e=\{i_1,\ldots,i_m\} \in E(G)$ (or equivalently $a_{i_1,\ldots,i_m} \ne 0$), then
\[
\frac{j}{\ell} \equiv \frac{\phi(i_1)+\cdots+\phi(i_m)}{r_{[p]}p} \mod \mathbb{Z}.
\]
If the vertices contained in $e$ receive the same color from $\phi$, as $( r_{[p]}p) \mid m \phi(i_1)$ by Lemma \ref{j-sigma},
we have
\[\frac{j}{\ell} \equiv \frac{ m \phi(i_1)}{r_{[p]}p} \equiv 0 \mod \mathbb{Z}, \]
which yields that $\frac{j}{\ell}$ is an integer, a contradiction.
So $G$ is $r_{[p]}p$-colorable.

Similarly, as $D^m=I$ for any $D  \in \Dg$ by Lemma \ref{group}(3), by (\ref{j-sigma2}) in Lemma \ref{j-sigma}, we have
\[ \frac{1}{\ell} \equiv \frac{\phi(i_1)+\cdots+\phi(i_m)}{m} \mod \mathbb{Z},
\]
yielding $G$ has an $(m,\ell)$-coloring, and hence $G$ is $m$-colorable.
The last result is obtained immediately from the above discussion.
\end{proof}

Finally we will investigate the structure of weakly irreducible nonnegative tensor, i.e. the zero entries distribution.
When applying to the hypergraphs, we will get the information of edge distribution.

\begin{coro}\label{stru1}
Let $\A$ be an $m$-th order $n$-dimensional weakly irreducible nonnegative tensor with $s(\A)=r \ge 2$.
Then there exists a partition of $[n]=V_1\cup \cdots \cup V_s$ for some integer $s \ge 2$ and a map $\phi: [n] \to [r]$ satisfying
$\phi|_{V_i}=l_i$ for $i \in [s]$, where $l_1(=r),l_2, \ldots, l_s$ are distinct integers,
such that
if $m l_{i_1}\not\equiv (l_{i_1}+\cdots+l_{i_m}) \mod r$, then
\begin{equation} \label{center1} \A[V_{i_1}|V_{i_2}|\cdots|V_{i_m}]=0.\end{equation}
Furthermore, if $\A$ is also symmetric, then
if $r\nmid  (l_{i_1}+\cdots+l_{i_m}) $, then
\begin{equation}\label{center2}\A[V_{i_1}|V_{i_2}|\cdots|V_{i_m}]=0;\end{equation}
or there exists a partition of $[n]=U_1\cup \cdots \cup U_t$ for some integer $t \ge 2$ and a map $\psi: [n] \to [m]$ satisfying
$\psi|_{U_i}=q_i$ for $i \in [t]$, where $q_1(=m),q_2,\ldots, q_t$ are distinct integers, such that if
$m \nmid  (q_{i_1}+\cdots+q_{i_m}) $, then
 \begin{equation}\label{center2-sym}\A[U_{i_1}|U_{i_2}|\cdots|U_{i_m}]=0.\end{equation}
\end{coro}

\begin{proof}
By Lemma \ref{group}(1), as $r \ge 2$, choose $D \in \Dg^{(0)}$ such that $D \ne I$ and $D^r=I$.
By (\ref{j-sigma1}) of Lemma \ref{j-sigma}, taking $j=0$ and $\sigma=r$,
then there exists a partition of $[n]=V_1\cup \cdots \cup V_s$ and a map $\phi: [n] \to [r]$ satisfying
$\phi|_{V_i}=l_i$ for $i \in [s]$, where $l_1(=r),l_2, \ldots, l_s$ are distinct integers,
such that if $a_{i_1 \ldots i_m} \ne 0$ for $i_1 \in V_{i_1}, \ldots, i_m \in V_{i_m}$,
\[
\frac{m l_{i_1}}{r} \equiv \frac{l_{i_1}+\cdots+l_{i_m}}{r} \mod \mathbb{Z},
\]
implying (\ref{center1}).
If $\A$ is also symmetric, by (\ref{j-sigma2}) of Lemma \ref{j-sigma},
we get $r\mid  (l_{i_1}+\cdots+l_{i_m}) $ if $a_{i_1 \ldots i_m} \ne 0$ for $i_1 \in V_{i_1}, \ldots, i_m \in V_{i_m}$, yielding (\ref{center2}).
As $D^m=I$ for all $D \in \Dg$, we have a $D \in \Dg^{(0)}$ such that $D \ne I$ and $D^m=I$.
The remaining discussion is similar by using (\ref{j-sigma2}) and taking $j=0$ and $\sigma=m$.
\end{proof}

\begin{coro} \label{stru2}
Let $\A$ be an $m$-th order $n$-dimensional weakly irreducible nonnegative tensor with $s(\A)=r$,
which is spectral $\ell$-symmetric ($\ell \ge 2)$.
Then there exists a partition of $[n]=V_1\cup \cdots \cup V_s$ for some integer $s \ge 2$ and
a map $\phi: [n] \to [r \ell]$ satisfying
$\phi|_{V_i}=l_i$ for $i \in [s]$, where $l_1(=r \ell),l_2, \ldots, l_s$ are distinct integers,
such that
if $r+m l_{i_1} \not\equiv (l_{i_1}+\cdots+l_{i_m}) \mod r\ell$, then
\begin{equation} \label{rot1} \A[V_{i_1}|V_{i_2}|\cdots|V_{i_m}]=0.\end{equation}
Furthermore, if $\A$ is also symmetric, if $r \not\equiv (l_{i_1}+\cdots+l_{i_m}) \mod r\ell$, then
\begin{equation}\label{rot21}\A[V_{i_1}|V_{i_2}|\cdots|V_{i_m}]=0,\end{equation}
or there exists a partition of $[n]=U_1\cup \cdots \cup U_t$ for some integer $t \ge 2$ and
a coloring $\psi: [n] \to [m]$ satisfying
$\psi|_{U_i}=q_i$ for $i \in [t]$, where $q_1(=m), q_2,\ldots, q_t$ are distinct integers,
such that
if $\frac{m}{\ell} \not\equiv (q_{i_1}+\cdots+q_{i_m}) \mod m$, then
\begin{equation}\label{rot22}\A[U_{i_1}|U_{i_2}|\cdots|U_{i_m}]=0.\end{equation}
\end{coro}

\begin{proof}
By Lemma \ref{group}(2), there exists a diagonal $D \in \Dg^{(1)}$ with $D^{r \ell}=I$.
Taking $j=1$ and $\sigma=r \ell$ in Lemma \ref{j-sigma}, we get the first result from (\ref{j-sigma1}) and the second result from (\ref{j-sigma2}).
Also, for any $D \in \Dg$, $D^m=I$.
So, taking $j=1$ and $\sigma=m$ in Lemma \ref{j-sigma}, we get the third result from (\ref{j-sigma2}).
\end{proof}

\begin{remark}
Let $\A$ be an $m$-th order $n$-dimensional weakly irreducible nonnegative tensor.
If $m=2$, then $\A$ is an irreducible nonnegative matrix $\A$, $s(\A)=1$ and $c(\A) \ge 1$.
If moreover $c(A) \ge 2$, then $\A$ has a structure as in (\ref{mat-stru1}).

Now for the case of $m \ge 3$, it will happen that $s(\A) \ge 2$ or $c(\A) \ge 2$.
If $s(\A) \ge 2$, then $\A$ has a structure as in (\ref{center1}), (\ref{center2}) or (\ref{center2-sym}).
If $c(\A) \ge 2$, then $\A$ has a structure as in (\ref{rot1}), (\ref{rot21}) or (\ref{rot22}).
This is the difference between low-dimensional tensors (matrices) and high-dimensional tensors.
\end{remark}

\begin{coro} \label{k-color}
Let $\A$ be an $m$-th order $n$-dimensional symmetric weakly irreducible nonnegative tensor, which is spectral $\ell$-symmetric ($\ell \ge 2)$.
Then $\A$ is $(m,\ell)$-colorable.
\end{coro}

\begin{proof}
By Corollary \ref{stru2} and (\ref{rot22}), we have a  map $\phi:[n] \to [m]$ such that (\ref{gen-col}) holds, implying that
$\A$ has  an $(m,\ell)$-coloring.  \end{proof}

\begin{lemma} \label{col-sym}
If an $m$-th order $n$-dimensional tensor $\A$ is $(m,\ell)$-colorable, then it is spectral $\ell$-symmetric.
\end{lemma}

\begin{proof}
Suppose that $\A$ has  an $(m,\ell)$-coloring $\phi: [n] \to [m]$.
Let \[D=\hbox{diag}\{e^{\i \frac{2 \phi(1)\pi}{m}}, \ldots,e^{\i \frac{2 \phi(n)\pi}{m}}\}.\]
It is easy to verify that
$\A=e^{-\i \frac{2 \pi}{\ell}}D^{-(m-1)}\A D$ by (\ref{gen-col}).
So $\A$ is spectral $\ell$-symmetric. \end{proof}

The following two theorems follow by Corollary \ref{k-color} and Lemma \ref{col-sym} immediately, which generalize the
Nikiforov's results on spectral $2$-symmetry.

\begin{theorem}
Let $\A$ be a symmetric weakly irreducible nonnegative tensor of order $m$.
Then $\A$ is spectral $\ell$-symmetric if and only if $\A$ is $(m,\ell)$-colorable.
\end{theorem}

\begin{theorem} \label{sym-col-graph}
Let $G$ be a connected $m$-uniform hypergraph.
Then $G$ is spectral $\ell$-symmetric if and only if $G$ is $(m,\ell)$-colorable.
\end{theorem}

\begin{example} \label{exa1}
Let $\A=(a_{ijk})$ be a tensor of order $3$ and dimension $6$ (cf. \cite{Ni}), where $i,j,k \in [6]$, such that
$ a_{123}=a_{234}=a_{345}=a_{456}=a_{561}=a_{612}=1,$
and all other entries are zero.
The eigen-equations of $\A$ are
\begin{equation}\label{eqn} \la x_1^2=x_2x_3, \la x_2^2=x_3x_4, \la x_3^2=x_4x_5,\la x_5^2=x_6x_1,\la x_6^2=x_1x_2.\end{equation}
If $\la \ne 0$, then $\la^6=1$, implying $\A$ is spectral $6$-cyclic (i.e. $c(\A)=6$).
If taking $x_1=1$ and $x_2$ as a parameter, then
by the first four equations of (\ref{eqn}) we get
\begin{equation}\label{eigenvector} x_1=1, x_2=x_2, x_3=\la x_2^{-1}, x_4=x_2^3, x_5=\la^3 x_2^{-5},x_6=\la^4 x_2^{11}.\end{equation}
From the 5th or 6th equation, we get $x_2^{21}=\la^3$.
So, if letting $\la=1, \tau=e^{\i 2\pi / 21}$ in (\ref{eigenvector}),
   then get $21$ different eigenvectors $x^{(0j)}$ ($j \in [21]$) corresponding $\rho(\A)$ listed in Table \ref{tab1}, implying $s(\A)=21$.

Let $\omega=e^{\i 2\pi / 6}$ and $\xi=e^{\i \pi/21}$.
Then by a similar discussion, for each $i \in [5]$, we  get $21$ different eigenvectors $x^{(ij)}$ ($j \in [21]$) corresponding to the eigenvalue $\omega^i$ listed in Table 3.1.

\begin{table}[ht]
\caption{Eigenvalue and eigenvectors of the tensor in Example \ref{exa1}}\label{tab1}
\renewcommand\arraystretch{1.5}
\noindent\[
\begin{array}{|c|l|}
\hline
\mbox{Eigenvalues} & \mbox{Eigenvectors}\\
\hline 1 & x^{(0j)}=(1,\tau^j,(\tau^j)^{-1},(\tau^j)^3,(\tau^j)^{-5},(\tau^j)^{11}) \\
\hline \omega & x^{(1j)}=(1,\tau^j \xi, \omega (\tau^j \xi)^{-1},(\tau^j \xi)^3, \omega^3 (\tau^j \xi)^{-5}, \omega^4 (\tau^j \xi)^{11})\\
\hline
 \omega^2 & x^{(2j)}=(1,\tau^j , \omega^2 (\tau^j)^{-1},(\tau^j )^3, (\tau^j )^{-5}, \omega^{2} (\tau^j )^{11}) \\
\hline
\omega^3 & x^{(3j)}=(1,\tau^j \xi, \omega^3 (\tau^j \xi)^{-1},(\tau^j \xi)^3, \omega^3 (\tau^j \xi)^{-5},  (\tau^j \xi)^{11}) \\
\hline
\omega^4 & x^{(4j)}=(1,\tau^j, \omega^4 (\tau^j )^{-1},(\tau^j )^3,  (\tau^j)^{-5}, \omega^4 (\tau^j )^{11}) \\
\hline
\omega^5 & x^{(5j)}=(1,\tau^j \xi, \omega^5 (\tau^j \xi)^{-1},(\tau^j \xi)^3, \omega^3 (\tau^j \xi)^{-5}, \omega^2 (\tau^j \xi)^{11})  \\
\hline
\end{array}
\]
\end{table}

For each $i=0,1,\ldots,5$ and each $j \in [21]$, we associated the eigenvector $x^{(ij)}$ with a diagonal matrix $D^{(ij)}=\hbox{diag}(x^{(ij)})=\hbox{diag}(x^{(ij)}_1,\ldots,x^{(ij)}_6)$, and form a set
$\Dg^{(i)}=\{D^{(i1)},\ldots,D^{(i,21)}\}$.
By Lemma \ref{group}, $\Dg=\cup_{i=0}^5 \Dg^{(i)}$ is a group of order $126$, and $\Dg^{(0)}$ is a subgroup of order $21$.
Each set $ \Dg^{(i)}$ ($i \in [5]$) is a coset of $\Dg^{(0)}$.
For example, one can verify $D^{(1,j+1)}=D^{(0j)} D^{(11)}$ for $j \in [21]$, where the superscript is taken modulo $21$, so that $\Dg^{(1)}=\Dg^{(0)}D^{(11)}$.

Writing \[
D^{(01)}=\hbox{diag}\left(1=e^{\i \frac{2 \pi {\bf 21}}{21}},e^{\i \frac{2 \pi {\bf 1}}{21}},e^{\i \frac{2 \pi {\bf 20}}{21}},
e^{\i \frac{2 \pi {\bf 3}}{21}},e^{\i \frac{2 \pi {\bf 16}}{21}},e^{\i \frac{2 \pi {\bf 11}}{21}}\right).\]
Then we have a partition $[n]=V_1 \cup \cdots \cup V_6$, where $V_i=\{i\}$ for $i \in [6]$; and a map $\phi_1: [n] \to [21]$, such that
\[\phi_1(1)=21, \phi_1(2)=1,\phi_1(3)=20,\phi_1(4)=3,\phi_1(5)=16,\phi_1(6)=11.\]
One can verify that (\ref{center1}) of Corollary \ref{stru1} holds.
For example,  as the solutions of the equation
\[
3 \phi_1(1)  \equiv  (\phi_1(1)+\phi_1(j)+\phi_1(k)) \mod 21
\]
are $(j,k)=(2,3)$ and  $(j,k)=(3,2)$,
except $a_{123}$ and $a_{132}$, all other entries $a_{1jk}=0$ for $j,k \in [6]$.

As $2 \mid 6$ and $2 \nmid 21$, by Lemma \ref{group}(2), there exists a $D \in \Dg\backslash \Dg^{(0)}$ such that $D^2=I$.
Such $D$ is contained in $\Dg^{(3)}$, i.e.
\[ D=\hbox{diag}\{x^{(3,10)}\}=\hbox{diag}\{1,-1,1,-1,1,-1\}.    \]
\end{example}

\begin{example}\label{exa2}
Next we consider a symmetrization form of the tensor $\A$ in Example \ref{exa1}.
Let $G$ be a $3$-uniform hypergraph with vertex set $[6]$ and edge set
$$\{ \{1,2,3\},\{2,3,4\},\{3,4,5\},\{4,5,6\},\{5,6,1\},\{6,1,2\}\}.$$
Let $\A=A(G))$ be the adjacency tensor of $G$.
Since $G$ is $3$-regular, $\rho(G)=3$, with the all-one vector as an eigenvector.
By Lemma \ref{ev}, for any eigenvector $x$ corresponding to an eigenvalue with modulus $\rho(G)$, $|x|=x^{(01)}$ if normalizing $x_1=1$;
and by Corollary \ref{m-ev}(4), $x^{[3]}=x^{(01)}$.
Let $\omega=e^{\i \frac{2 \pi}{3}}$.
By the eigenvector equations of $\A(G)$,
the eigenvalues with modulus $\rho(G)$ are $\la_k=3 \omega^k \;(k=0,1,2)$, which are corresponding to the eigenvectors $x^{(kj)}\; (j \in [3])$ listed as in Table \ref{tab2},
that is, $s(G)=3$ and $c(G) =3$.

\begin{table}[ht]
\caption{Eigenvalue and eigenvectors of the tensor in Example \ref{exa2}}\label{tab2}
\renewcommand\arraystretch{1.5}
\noindent\[
\begin{array}{|c|l|}
\hline
\mbox{Eigenvalues} & \mbox{Eigenvectors}\\
\hline 3 & x^{(0j)}=(1,\omega^j,\omega^{-j},1,\omega^j,\omega^{-j}) \\
\hline 3\omega & x^{(1j)}=(1,\omega^j,\omega^{1-j},1,\omega^j,\omega^{1-j})\\
\hline
 3\omega^2 & x^{(2j)}=(1,\omega^j,\omega^{2-j},1,\omega^j,\omega^{2-j}) \\
\hline
\end{array}
\]
\end{table}

For each $i=0,1,2$ and each $j \in [3]$, define $D^{(ij)}=\hbox{diag}(x^{(ij)})$, and form a set
$\Dg^{(i)}=\{D^{(i1)},D^{(i2)},D^{(i,3)}\}$.
By Lemma \ref{group}, $\Dg=\cup_{i=0}^2 \Dg^{(i)}$ is a group of order $9$, and $\Dg^{(0)}$ is a subgroup of order $3$.
It is easy to verify that each set $ \Dg^{(i)}$ ($i \in [2]$) is a coset of $\Dg^{(0)}$.
Also, we find that
$$ \Dg \cong \mathbb{Z}_3 \oplus \mathbb{Z}_3, \Dg^{(0)} \cong \mathbb{Z}_3.$$
So each elementary divisor of $\Dg$ and $\Dg^{(0)}$ divides $m$ (here $m=3$).

From the eigenvector
\[ x^{(01)}=(e^{\i \frac{2 \pi \bf{3}}{3}},e^{\i \frac{2 \pi \bf{1}}{3}},e^{\i \frac{2 \pi \bf{2}}{3}},
 e^{\i \frac{2 \pi \bf{3}}{3}},e^{\i \frac{2 \pi \bf{1}}{3}},e^{\i \frac{2 \pi \bf{2}}{3}})
 ,\]
we have a map $\phi_1: [6] \to [3]$ such that $\phi_1|_{\{1,4\}}=3$, $\phi_1|_{\{2,5\}}=1$ and $\phi_1|_{\{3,6\}}=1$ holding (\ref{center2}).
Let $V_1=\{1,4\}$, $V_2=\{2,5\}$ and $V_3=\{3,6\}$.
Observe that the equation
\[0 \equiv (\phi_1(V_i)+\phi_1(V_j)+\phi_1(V_k) \mod 3,
\]
has solutions $(i,j,k)$ with $\{i,j,k\}=\{1,2,3\}$.
So, we have a partition $V(G)=V_1 \cup V_2 \cup V_3$ such that each edge of $G$ intersects those three parts.
Equivalently, if one of $i,j,k$ occurs more than one time, then
\[\label{distrb1}
\A(G)[V_i|V_j|V_k]=0,
\]
which implies that $G$ contains no edges
$\{1,2,4\}$,$\{1,2,5\}$, $\{1,3,4\}$, $\{1,3,6\}$, $\{1,4,5\}$, $\{1,4,6\}$,
$\{2,3,5\}$, $\{2,3,6\}$, $\{2,4,5\}$, $\{2,5,6\}$, $\{3,4,6\}$, $\{3,5,6\}$.

From the eigenvector
\[x^{(11)}=(e^{\i \frac{2 \pi \bf{3}}{3}},e^{\i \frac{2 \pi \bf{1}}{3}},
 e^{\i \frac{2 \pi \bf{3}}{3}},e^{\i \frac{2 \pi \bf{3}}{3}},e^{\i \frac{2 \pi \bf{1}}{3}},e^{\i \frac{2 \pi \bf{3}}{3}}),\]
we have a map $\phi_2: [6] \to [3]$ such that $\phi_2|_{\{1,3,4,6\}}=3$ and $\phi_2|_{\{2,5\}}=1$ holding (\ref{rot22}).
Let $U_1=\{1,3,4,6\}$ and $U_2=\{2,5\}$.
Then every edge of $G$ takes two vertices from $U_1$ and one from $U_2$.
So we get the information of edge distribution of $G$, in particular we find almost all non-edges of $G$ in this example.
\end{example}

\section{Cyclic index of tensors and hypergraphs}
In this section, we will discuss how to characterize the spectral symmetry or the cyclic index of weakly irreducible nonnegative tensors or connected hypergraphs.
Shao et.al give a characterization of spectral $m$-symmetric $m$-uniform hypergraphs by using the generalized traces; see \cite[Theorem 3.1]{SQH}.
Following their idea, we get a generalized result.

\begin{theorem} \label{ellsym}
Let $\A$ be a tensor of order $m$ and dimension $n$, and $\varphi_\A(\la)=\sum_{i=0}^D a_i \la^{D-i}\;(D=n(m-1)^{n-1})$ be the characteristic polynomial of $\A$.
Then the following conditions are equivalent.

\begin{enumerate}

\item  $\A$ is spectral $\ell$-symmetric.

\item If $\ell \nmid d$, then $a_d=0$, i.e. $\varphi_\A(\la)=\la^t f(\la^\ell)$ for some nonnegative integer $t$ and some polynomial $f$.

\item If $\ell \nmid d$, then $\Tr_d(\A)=0$.

\end{enumerate}
\end{theorem}

\begin{proof}
$(1) \Longrightarrow (2)$. Let $\epsilon=e^{\i \frac{2\pi}{\ell}}$.
Then (1) implies that $\psi_\A(\epsilon\la)=\epsilon^D\varphi_\A(\la)$.
So \begin{equation} \label{trans} \sum_{d=0}^D a_d \epsilon^{D-d}\la^{D-d}=\epsilon^D \sum_{d=0}^D a_d \la^{D-d}.\end{equation}
Then we have $a_d(\epsilon^d-1)=0$.
So, if $\ell \nmid d$, then $\epsilon^d-1\neq 0$, and hence $a_d=0$.

It is easily seen that $(2) \Longrightarrow (1)$.

$(2) \Longrightarrow (3)$. From (2) we have for some integer $s$,
\[\varphi_\A(\la)=\la^t(\la^\ell-c_1^\ell)\cdots (\la^\ell-c_s^\ell).\]
Let $P$ be a circulant permutation matrix of dimension $\ell$, that is $p_{ij}=1$ if and only if $j\equiv i+1 \mod \ell$.
Then $\psi_{cP}(\la)=\la^\ell - c^\ell$.
If letting $B=c_1P \oplus \cdots \oplus c_sP$, then
$\varphi_\A(\la)=\la^t \psi_B(\la)$, and by (\ref{traceF})
\[ \Tr_d(\A)=\Tr(B^d)=\Tr((c_1P)^d)+\cdots+\Tr((c_sP)^d).\]
It is known that if $\ell \nmid  d$, then $\Tr((c_iP)^d)=0$ for $i \in [s]$, and hence $\Tr_d(\A)=0$.

$(3) \Longrightarrow (2)$. From (\ref{coef}) and (\ref{schur}),
if $a_d \neq 0$, then there exist some positive integers $d_1,\ldots,d_t$ with $d_1+\cdots+d_t=d$ such that
\begin{equation}\label{sum-trace}\Tr_{d_1}(\A)\cdots \Tr_{d_t}(\A) \neq 0.\end{equation}
By the condition (3), we know that $\ell \mid  d_i$ for $i \in [t]$, yielding $\ell \mid d$.
\end{proof}

\begin{coro}\label{div}
Let $\A$ be a tensor of order $m$ and dimension $n$, and $\varphi_\A(\la)=\sum_{i=0}^D a_i \la^{D-i}\;(D=n(m-1)^{n-1})$ be the characteristic polynomial of $\A$.
If $\A$ is spectral $\ell$-symmetric, then
\[ \ell \mid  \hbox{g.c.d.}\{d: a_d \ne 0\}, \; \ell \mid  \hbox{g.c.d.}\{d: \Tr_d(\A) \ne 0\}.\]
Furthermore,
\begin{equation} \label{gcd2} c(\A)=\hbox{g.c.d.} \{d: a_d \ne 0\}=\hbox{g.c.d.}\{d: \Tr_d(\A) \ne 0\}.\end{equation}
\end{coro}

\begin{proof}
The first result is obtained by the equivalence of (1) and (2) in Theorem \ref{ellsym}.
So, $ c(\A)\mid \hbox{g.c.d.} \{d: a_d \ne 0\}$.
Let $g:=\hbox{g.c.d.} \{d: a_d \ne 0\}$ and $\epsilon=e^{\i \frac{2\pi}{g}}$.
Then for all $d$ with $a_d\ne 0$, $g\mid d$ and  $a_d({e^{\i\frac{2\pi d}{g}}}-1)=0$.
So (\ref{trans}) holds and hence
$\psi_\A(\epsilon\la)=\epsilon^D\varphi_\A(\la)$,
which implies that $\A$ is spectral $g$-symmetric.
By the definition of $c(\A)$, $c(\A)=g=\hbox{g.c.d.} \{d: a_d \ne 0\}$.

Now let $\bar{g}:=\hbox{g.c.d.}\{d: \Tr_d(\A) \ne 0\}$.
From (\ref{sum-trace}), if $a_d \ne 0$, then  there exist some positive integer $d_1,\ldots,d_t$ with $d_1+\cdots+d_t=d$ such that $\Tr_{d_1}(\A)\cdots \Tr_{d_t}(\A) \neq 0$.
Then, $\bar{g} \mid  d_i$ for $i \in [t]$, and hence $\bar{g}\mid d$, which implies that $\bar{g}\mid g$.
On the other hand, by what we have proved, $g\mid \bar{g}$, yielding (\ref{gcd2}).
\end{proof}

In the remaining part of this paper, we will discuss the spectral symmetry of connected hypergraphs.
Cooper and Dutle \cite{CD} raised a problem on characterizing the $m$-uniform hypergraphs whose spectra are invariant under multiplication by the $m$-th roots
of unity (i.e. the spectral $m$-symmetric $m$-uniform hypergraphs by our definition).
Pearson and Zhang posed a more specific problem on characterizing all connected uniform hypergraphs with symmetric spectrum (i.e. the spectral $2$-symmetric $m$-uniform connected hypergraphs).
Nikiforov \cite{Ni} obtains a complete solution to the latter problem, which is exactly the result of Theorem \ref{sym-col-graph} for $\ell=2$.

Shao et.al \cite{SSW} give a characterization on the symmetry of $H$-spectra of hypergraphs, that is,
an $m$-uniform hypergraph $G$ has a symmetric $H$-spectrum if and only if $m$ is even and $G$ is odd-bipartite (or odd transversal).
They posed a problem that whether $\S(\L(G))=\S(\Q(G))$ can imply that $H\S(\L(G))=H\S(\Q(G))$, which is equivalently to ask
whether $\S(\A(G))=-\S(\A(G))$ can imply that $H\S(\A(G))=-H\S(\A(G))$.
Zhou et.al also pose a similar conjecture whether $-\rho(\A(G))$ being an eigenvalue of $G$ can imply that $m$ is even and $G$ is odd-bipartite.
In fact,  $-\rho(\A(G))$ is an eigenvalue of $G$ if and only if $G$ has a symmetric spectrum.
By Nikiforov's result  \cite{Ni}, this is equivalent to ask whether an odd-colorable hypergraph is odd transversal.
They construct two classes of hypergraphs to give a negative answer.
We also give a negative answer to the above problem in \cite{FKT} by constructing a class of non-odd-bipartite generalized power hypergraphs.

In general, for an $m$-uniform hypergraph $G$, as $\A(G)$ is symmetric, if $G$ is spectral $\ell$-symmetric, then $\ell\mid m$ by Lemma \ref{group}.
If $G$ is $m$-partite \cite{CD}, or hm-bipartite\cite{HQ}, or $p$-hm bipartite with $(p,m)=1$ \cite{SQH}, then $G$ is spectral $m$-symmetric.
The following result is proved in Lemma \ref{group}(3), which is re-proved by using the generalized traces as follows.

\begin{coro}\label{ell-m}
Let $G$ be an $m$-uniform hypergraph on $n$ vertices.
If $G$ is spectral $\ell$-symmetric, then $\ell \mid m$.
\end{coro}

\begin{proof}
By Theorem 3.15 in \cite{CD}, the codegree $m$ coefficient of $\psi_G(\la)$ is
\[ \label{coeff} a_m=-m^{m-2}(m-1)^{n-m}|E(G)| \ne 0.\]
So, by Corollary \ref{div}, $\ell \mid m$.\end{proof}

A {\it simplex} in an $m$-uniform hypergraph is a set of $m+1$ vertices where every set of $m$ vertices forms an edge.

\begin{coro}\label{simplex}
Let $G$ be an $m$-uniform hypergraph on $n$ vertices.
If $G$ contains a simplex, then $G$ is  spectral $1$-cyclic or nonsymmetric.
\end{coro}

\begin{proof}
By Theorem 3.17 in \cite{CD}, the codegree $m+1$ coefficient of $\psi_G(\la)$ is
 \[ a_{m+1}=-C(m-1)^{n-m}s,\]
  where $s$ is the number of simplices in $G$,
  and $C$ is a positive integer depending only on $m$.
 By Corollary \ref{div}, $c(\A(G))=\hbox{g.c.d}\{m,m+1,\ldots\}=1$. \end{proof}

The following result in the case of $(p,m)=1$ is given by Shao et.al \cite{SQH}.

\begin{theorem} \label{gener}
Let $G$ be an $m$-uniform $p$-hm bipartite hypergraph.
Then $G$ is spectral $\frac{m}{(p,m)}$-symmetric.
\end{theorem}

\begin{proof}
By Lemma \ref{ellsym}, it suffices to prove that if $\Tr_d(G) \ne 0$, then $\frac{m}{(p,m)}\mid d$.
Suppose that $\Tr_d(G) \ne 0$.
By (\ref{tracedF}) and (\ref{F}), there exists $F=\{e_1(i_1),\ldots,e_d(i_d)\} \in \F_d(G)$.
Let $H$ be the sub-hypergraph of $G$ induced by those $d$ edges $e_1,\ldots,e_d$.
Then the degree $d_H(v)$ of each vertex $v$ of $H$ is $m$-valent.
Let $[n]=V_1 \cup V_2$ be a partition of the vertex set $V(G)$ such that each edge intersects $V_1$ with exactly $p$ vertices.
Now \[ \sum_{v \in V_1}d_H(v)=dp.\]
As each vertex of $H$ has an $m$-valent degree, $m \mid  dp$, which implies that $\frac{m}{(p,m)}\mid d$. \end{proof}

 Finally we will discuss the spectral symmetry of generalized power hypergraphs, and
show that for an arbitrarily given positive integer $m$ and any positive integer $\ell $ with $\ell |m$, there exists an $m$-uniform hypergraph $G$ with $c(G)=\ell$.

\begin{lemma} \label{powerB}
Let $G$ be a simple bipartite graph, and let $m \ge 4$ be an even integer.
Then $G^{m,\frac{m}{2}}$ is spectral $m$-cyclic.
\end{lemma}

\begin{proof}
By Corollaries \ref{div} and \ref{ell-m}, it suffices to prove that if $\Tr_d(G^{m,\frac{m}{2}}) \ne 0$, then $m \mid d$.
Let $V(G) =V_1 \cup V_2$ be a bipartition of $V(G)$ such that each edge of $G$ intersects both $V_1$ and $V_2$,
 which naturally corresponds a bipartition $V(G^{m,\frac{m}{2}})=\mathbf{V}_1 \cup \mathbf{V}_2$,
 where $\mathbf{V}_i$ is obtained from $V_i$ by blowing each vertex $v$ into an $\frac{m}{2}$-set $\v$ for $i \in [2]$.
So we may assume that $V_i \subset \mathbf{V}_i$ for $i \in [2]$.
Suppose that $\Tr_d(G^{m,\frac{m}{2}}) \ne 0$.
By (\ref{tracedF}) and (\ref{F}), we have $d$ edges
$F=\{e_1(i_1), \ldots, e_d(i_d)\} \in \F_d(G^{m,\frac{m}{2}})$.
Let $H$ be the sub-hypergraph of $G^{m,\frac{m}{2}}$ induced by the edges in $F$.
Then
\[
\sum_{v \in V_1 \cap V(H)} d_H(v)=d.
\]
As each vertex in $H$ is $m$-valent, we have $m \mid d$.
 \end{proof}

\begin{lemma} \label{equiv}
Let $G$ be a $t$-uniform hypergraph and let $G^{m,\frac{m}{t}}$ be the generalized power of $G$, where $t|m$.
Then $c(G^{m,\frac{m}{t}})=\frac{m}{t} \cdot l$, where $l$ is a positive integer such that $l \mid c(G)$.
\end{lemma}

\begin{proof}
Let $s:=\frac{m}{t}$.
Note that each edge $e=\{i_1,\ldots,i_t\}$ of $G$ is $1$-$1$ corresponding to the edge
  $\bar{e}=S_{i_1} \cup \cdots S_{i_t}$ of $G^{m,s}$, where $S_{i_j}=\{i_{j1},\ldots,i_{js}\}$ for $j \in [t]$.
Hence we give a labeling of the vertices of $G^{m,s}$.
 So, if we choose $i_{11},\ldots,i_{t1}$ from each edge $\bar{e}$ of $G^{m,s}$ with the above expression and form a set $U$, then every edge of $G^{m,s}$
  intersects $U$ with exactly $t$ vertices, which implies that $G^{m,s}$ is a $t$-hm bipartite hypergraph.
 By Theorem \ref{gener}, $G^{m,s}$ is spectral $\frac{m}{(m,t)}$-symmetric or spectral $s$-symmetric.
By Corollary \ref{div}, if $\Tr_{\bar{d}}(G^{m,s}) \ne 0$, then $s|\bar{d}$.
So, it is enough to consider the trace $\Tr_{ds}(G^{m,s})$ for a general positive integer $d$.

Suppose $\Tr_d(G) \ne 0$.
By (\ref{tracedF}) and (\ref{F}), each nonzero (positive) term of $\Tr_d(G)$ is associated with $d$ edges such that
$$ F=\{e_1(i_1), \ldots, e_d(i_d)\} \in \F_d(G),$$ and $D(F)$ is a Eulerian directed graph, the latter of which also implies that
each index occurring in $F$ will be a primary index and a secondary index as well.
Correspondingly, we have $ds$ edges of $G^{m,s}$ such that
\[\label{construct}\bar{F}=\{\bar{e}_{1}(i_{11}), \ldots, \bar{e}_{1}(i_{1s}),\ldots,\bar{e}_d(i_{d1}), \ldots, \bar{e}_d(i_{ds})\} \in \F_{ds}(G^{m,s}).\]
Here each edge $\bar{e}_j$ occurs $s$ times but with different primary index for $j \in [d]$.

Now we consider the directed graph $D(\bar{F})$.
For each fixed $j \in [s]$, let $D_j(\bar{F})$ be the subgraph of $D(\bar{F})$ induced by all possible vertices $i_{pj}$ if $i_p$ occurs in $F$,
 or equivalently, $D_j(\bar{F})$ uses the edges arising from $F_j=\{\bar{e}_1(i_{1j}), \ldots, \bar{e}_d(i_{dj})\}$ and the vertices of form $i_{pj}$ contained in those edges.
The vertex sets of $D_1(\bar{F}), \ldots, D_t(\bar{F})$ form a partition of $D(\bar{F})$, i.e.
\[V(D(\bar{F}))=V(D_1(\bar{F})) \cup \cdots \cup V(D_s(\bar{F})).\]
By the construction of $D(F),D(\bar{F})$ and $G^{m,s}$, for each $j \in [s]$,
  $D_j(\bar{F})$ is isomorphic to $D(F)$, which implies that $D_j(\bar{F})$ contains a Eulerian directed circuit.

For any fixed $j \in [s]$, we assert that in the directed graph $D(\bar{F})$, each vertex of $D_j(\bar{F})$ has the same number of in-neighbors and out-neighbors outside $D_j(\bar{F})$.
We will prove the assertion by considering each primary index $i_{pj}$ occurring in $\bar{F}$ for $p \in [d]$, including multiplicity, which means
if considering $i_{pj}$ as a vertex, the (out- or in-)degree of $i_{pj}$ is the sum of (out- or in-)degrees of the index $i_{pj}$ for each appearance of $i_{pj}$ (thinking of all $i_{pj}$ being distinct).
Observe that $i_{pj}$ has $t(s-1)$ out-neighbors outside $D_j(\bar{F})$ arising from the edge $\bar{e}_p(i_{pj})$,
where for each $l \in [s]\backslash \{j\}$, $i_{pj}$ has $t$ out-neighbors in $D_l(\bar{F})$.
As $D_j(\bar{F})$  contains a Eulerian directed circuit and $i_{pj}$ has exactly $t-1$ out-neighbors in $D_j(\bar{F})$ arising from the edge $\bar{e}_p(i_{pj})$,
 $i_{pj}$ also has exactly $t-1$ in-neighbors in $D_j(\bar{F})$, say $i_{q_1j},\ldots,i_{q_{t-1}j}$,
   corresponding to $t-1$ edges $\bar{e}_{q_1}(i_{q_1j}),\ldots,\bar{e}_{q_{t-1}}(i_{q_{t-1}j})$, each of which contains $i_{pj}$ as a secondary index.
So, by the construction of $\bar{F}$, for each $l \in [t-1]$, we have $s-1$ edges $\bar{e}_{q_l}(i_{q_l\tilde{j}})$'s for $\tilde{j} \in [s]\backslash \{j\}$,
  each of which also contains $i_{pj}$ as a secondary index.
In addition, for each $\tilde{j} \in [s]\backslash \{j\}$, the edge $\bar{e_p}(i_{p\tilde{j}})$'s contains $i_{pj}$ as a secondary index.
So, $i_{pj}$ has $(t-1)(s-1)+(s-1)=t(s-1)$ in-neighbors outside $D_j(\bar{F})$.

By the above discussion, we know $D(\bar{F})$ is weakly connected (considering its underlying undirected graph), and
each vertex of $D(\bar{F})$ has the same in-degree and out-degree, which implies that
 $D(\bar{F})$ is Eulerian, and hence $\Tr_{ds}(G^{m,s}) \ne 0$ by (\ref{tracedF}).
By Corollary \ref{div},
\begin{equation}
c(G^{m,s})=\hbox{g.c.d.} \{ds: \Tr_{ds}(G^{m,s}) \ne 0\}= s \cdot \hbox{g.c.d.} \{d: \Tr_{ds}(G^{m,s}) \ne 0\}.
\end{equation}
Let $l:=\hbox{g.c.d.} \{d: \Tr_{ds}(G^{m,s}) \ne 0\}$.
By what we have proved, if $\Tr_d(G) \ne 0$, then $\Tr_{ds}(G^{m,s}) \ne 0$, implying that $l | c(G)$.
 \end{proof}

If $c(G)=1$, then $c(G^{m,s})=s \cdot c(G)$ by Lemma \ref{equiv}. We pose the following conjecture.

\begin{conj}\label{conj}
\begin{equation} c(G^{m,s})=s \cdot c(G). \end{equation}
\end{conj}


\begin{coro} \label{powerNB}
Let $G$ be a simple non-bipartite graph, and let $m \ge 4$ be an even integer.
Then $G^{m,\frac{m}{2}}$ is spectral $\frac{m}{2}$-cyclic.
\end{coro}

\begin{proof}
As $G$ is non-bipartite, $c(G)=1$.
By Lemma \ref{equiv}, $c(G^{m,\frac{m}{2}})=\frac{m}{2} \cdot l$, where $l \mid c(G)$.
So $l=1$, and hence $c(G^{m,\frac{m}{2}})=\frac{m}{2}$.
 \end{proof}

\begin{coro}
Let $m \ge 2$ be an integer.
Then for any positive integer $\ell$ with $\ell |m$, there exists an $m$-uniform hypergraph $G$ such
that $G$ is spectral $\ell$-cyclic.
\end{coro}

\begin{proof}
Let $m =\ell t$.
If $\ell=1$, any $m$-uniform hypergraph with a simplex is as desired by Corollary \ref{simplex}.
If $\ell=m$, then the hypgergraph $G^{m,\frac{m}{2}}$ by taking $G$ be any bipartite graph is spectral $m$-cyclic by Corollary \ref{powerB}.
Otherwise, $2 \le t \le \frac{m}{2}$.
Let $G$ be a $t$-uniform hypgergraph (maybe simple graph) containing a simplex.
Then $G$ is spectral $1$-cyclic by Corollary \ref{simplex}.
By Theorem \ref{equiv}, $G^{m,\ell}$ is spectral $\ell$-cyclic. \end{proof}

\section{Conclusions}

It is known that the Earth has two kinds of movement: one is the rotation around its own axis, the other is the travel around the Sun.
Let $\A$ be a weakly irreducible nonnegative tensor of order $m$ and dimension $n$.
Governed by the group $\Dg$ defined in (\ref{Dg}) (taking $\ell=c(\A)$),
the ``movement'' of $\A$ has a similar behavior: one is self-rotating determined by $\Dg^{(0)}$ with ``period'' $s(\A)$, the other is
traveling through the orbit $S=\{e^{\i \frac{2 \pi i}{\ell}}\A: i=0,1,\ldots,c(\A)-1\}$ determined by $\Dg \backslash \Dg^{(0)}$ with ``period'' $c(\A)$.
So, we have two important parameters $s(\A)$ and $c(\A)$, from which we know the structural information of $\A$.
If $m=2$ or $\A$ is an irreducible nonnegative matrix, then $s(\A)=1$.
But for the case of $m \ge 3$, it will happen that $s(\A) \ge 2$; see Examples \ref{exa1} and \ref{exa2}.
So it will be interesting to investigate $s(\A)$ and get more structural information of $\A$.

%
%
%
%
%

\bibliographystyle{amsplain}

\end{document}